\documentclass[12pt,a4paper]{amsart}
\allowdisplaybreaks[1]

\usepackage[top=30truemm,bottom=30truemm,left=25truemm,right=25truemm]{geometry}
\usepackage{amstext,amsthm,amssymb}

\DeclareFontEncoding{OT2}{}{}
\DeclareFontSubstitution{OT2}{cmr}{m}{l}

\DeclareFontFamily{OT2}{cmr}{\hyphenchar\font45 }
\DeclareFontShape{OT2}{cmr}{m}{l}{%
<5><6><7><8><9>gen*wncyr%
<10><10.95><12><14.4><17.28><20.74><24.88>wncyr10}{}

\DeclareMathAlphabet{\mathcyr}{OT2}{cmr}{m}{l}
\DeclareMathAlphabet{\mathcyb}{OT2}{cmr}{b}{l}
\SetMathAlphabet{\mathcyr}{bold}{OT2}{cmr}{b}{l}

\newtheorem{thm}{Theorem}[section]
\newtheorem{lem}[thm]{Lemma}

\newcommand{\sha}{\mathbin{\widetilde{\mathcyr{sh}}}}
\newcommand{\sh}{\mathbin{\mathcyr{sh}}}

\begin{document}

\title{Bowman-Bradley type theorem for finite multiple zeta values in $\mathcal{A}_2$}

\author{Hideki Murahara}
\address[Hideki Murahara]{Nakamura Gakuen University Graduate School, 5-7-1, Befu, Jonan-ku,
Fukuoka, 814-0198, Japan}
\email{hmurahara@nakamura-u.ac.jp}

\author{Tomokazu Onozuka}
\address[Tomokazu Onozuka]{Multiple Zeta Research Center, Kyushu University 744, Motooka, Nishi-ku,
Fukuoka, 819-0395, Japan}
\email{t-onozuka@math.kyushu-u.ac.jp}

\author{Shin-ichiro Seki}
\address[Shin-ichiro Seki]{Mathematical Institute, Tohoku University, 6-3, Aoba, Aramaki, Aoba-Ku, Sendai, 980-8578, Japan}
\email{shinichiro.seki.b3@tohoku.ac.jp}

\thanks{The third author is supported in part by the Grant-in-Aid for JSPS Fellows (JP18J00151), The Ministry of Education, Culture, Sports, Science and Technology, Japan.}


\begin{abstract}
Bowman and Bradley obtained a remarkable formula among multiple zeta values. 
The formula states that the sum of multiple zeta values for indices which consist of the shuffle of two kinds of the strings
$\{1,3,\ldots,1,3\}$ and $\{2,\ldots,2\}$ is a rational multiple of a power of $\pi^2$.
Recently, Saito and Wakabayashi proved that analogous but more general sums of finite multiple zeta values in an adelic ring $\mathcal{A}_1$ vanish.
In this paper, we partially lift Saito-Wakabayashi's theorem from $\mathcal{A}_1$ to $\mathcal{A}_2$. Our result states that a Bowman-Bradley type sum of finite multiple zeta values in $\mathcal{A}_2$ is a rational multiple of a special element and this is closer to the original Bowman-Bradley theorem. 
\end{abstract}

\maketitle

%%%%%%%%%%%%%%%%%%%%%%%%%%%%%%%%%%%%%%%%%%%%%%%%%%%%%%%%%%%%%%%%%%%%%%%%%%%%%%%%%%%%%%%%%%%%%%%%%%%%%%
\section{Introduction}
For positive integers $k_1,\dots,k_r$ with $k_r\ge 2$, the multiple zeta values (MZVs) and the multiple zeta-star values (MZSVs) are defined by 
\begin{align*}
\zeta(k_1,\dots, k_r)&:=\sum_{1\le n_1<\cdots <n_r} \frac{1}{n_1^{k_1}\cdots n_r^{k_r}}, \\
\zeta^{\star}(k_1,\dots, k_r)&:=\sum_{1\le n_1\le \cdots \le n_r} \frac {1}{n_1^{k_1}\cdots n_r^{k_r}}. 
\end{align*} 
By convention, we set $\zeta(\varnothing)=\zeta^{\star}(\varnothing)=1$ for the empty index. Let $\{a_1,\ldots,a_l\}^m$ denote the $m$-times repetition of $a_1,\ldots,a_l$, e.g. $\{2\}^2=2,2$ and $\{1,3\}^2=1,3,1,3$. For MZVs, Bowman and Bradley \cite{BB02} established the following result:
\begin{thm}[{Bowman-Bradley \cite[Corollary 5.1]{BB02}}]
For non-negative integers $l$ and $m$, we have
\begin{align*}
&\sum_{\substack{ m_0+\cdots+m_{2l}=m \\ m_i\ge0\,(0\le i\le 2l) }}
\zeta (\{2\}^{m_0},1,\{2\}^{m_1},3,\{2\}^{m_2},\ldots,\{2\}^{m_{2l-2}},1,\{2\}^{m_{2l-1}},3,\{2\}^{m_{2l}})  \\
&=\binom{2l+m}{2l} \frac{\pi^{4l+2m}}{(2l+1)\cdot(4l+2m+1)!}.
\end{align*}
\end{thm}
A similar result for MZSVs is known by Kondo-Saito-Tanaka \cite{KST12} and Yamamoto \cite{Yam13}, i.e. the similar sum for MZSVs is also a rational multiple of $\pi^{4l+2m}$.

Let us consider counterparts of these results for finite multiple zeta values. For a positive integer $n$, we define the $\mathbb{Q}$-algebra $\mathcal{A}_{n}$ by
\[
\mathcal{A}_{n}
:=\biggl(\prod_{p}\mathbb{Z}/p^{n}\mathbb{Z}\biggr)\,\bigg/\, \biggl(\bigoplus_{p}\mathbb{Z}/p^{n}\mathbb{Z\biggr)},
\]
where $p$ runs over prime numbers. For positive integers $k_{1},\dots,k_{r}$
and $n$, the finite multiple zeta values (FMZVs) and the finite multiple zeta-star values (FMZSVs) in $\mathcal{A}_n$ are defined by 
\begin{align*}
\zeta_{\mathcal{A}_{n}}(k_{1},\dots,k_{r}) & :=\biggl(\sum_{1\le n_{1}<\cdots<n_{r}\le p-1}\frac{1}{n_{1}^{k_{1}}\cdots n_{r}^{k_{r}}}\bmod p^{n}\biggr)_{p}\in\mathcal{A}_{n}, \\
\zeta_{\mathcal{A}_{n}}^{\star}(k_{1},\dots,k_{r}) & :=\biggl(\sum_{1\le n_{1}\le \cdots\le n_{r}\le p-1}\frac{1}{n_{1}^{k_{1}}\cdots n_{r}^{k_{r}}}\bmod p^{n}\biggr)_{p}\in\mathcal{A}_{n}.
\end{align*}
We set $\zeta_{\mathcal{A}_n}(\varnothing)=\zeta_{\mathcal{A}_n}^{\star}(\varnothing)=1$. For details, see Rosen \cite{Ros15} and Seki \cite{Sek16}. Recently, Saito and Wakabayashi \cite{SW16} obtained Bowman-Bradley type results in a strong sense for finite multiple zeta values in $\mathcal{A}_1$. The following is a part of their results:
\begin{thm}[{Saito-Wakabayashi \cite[Theorem 1.4]{SW16}}] \label{SWthm}
Let $a$ and $b$ be odd positive integers and $c$ an even positive integer. For non-negative integers $l$ and $m$ with $(l,m)\neq(0,0)$, we have
\begin{align*}
&\sum_{\substack{ m_0+\cdots+m_{2l}=m \\ m_i\ge0\,(0\le i\le 2l) }}
\zeta_{\mathcal{A}_{1}} (\{c\}^{m_0},a,\{c\}^{m_1},b,\{c\}^{m_2},\ldots,\{c\}^{m_{2l-2}},a,\{c\}^{m_{2l-1}},b,\{c\}^{m_{2l}}) \\
&=\sum_{\substack{ m_0+\cdots+m_{2l}=m \\ m_i\ge0\,(0\le i\le 2l) }}
\zeta_{\mathcal{A}_{1}}^{\star} (\{c\}^{m_0},a,\{c\}^{m_1},b,\{c\}^{m_2},\ldots,\{c\}^{m_{2l-2}},a,\{c\}^{m_{2l-1}},b,\{c\}^{m_{2l}}) \\
&=0.
\end{align*}
\end{thm}
In this paper, we partially lift Saito-Wakabayashi's result from $\mathcal{A}_1$ to $\mathcal{A}_2$.
In fact, we show that the Bowman-Bradley type sum of FMZ(S)Vs in $\mathcal{A}_2$ for the shuffle of $\{1,3\}^l$ and $\{2\}^m$ is a rational multiple of the special element $\beta_{4l+2m+1}\boldsymbol{p}$. Here, $\boldsymbol{p}$ and $\beta_k$ are defined to be $(p \bmod{p^2})_p$ and  $(B_{p-k}/k \bmod{p^2})_p$ as elements of $\mathcal{A}_2$, respectively, where $B_n$ is the $n$th Seki-Bernoulli number and $k$ is an integer greater than $1$.
Then, our main theorem is the following:
\begin{thm}[Main theorem] \label{main}
For non-negative integers $l$ and $m$ with $(l,m)\neq(0,0)$, we have
\begin{equation} \label{MT1}
\begin{split}
&\sum_{\substack{ m_0+\cdots+m_{2l}=m \\ m_i\ge0\,(0\le i\le 2l) }}
\zeta_{\mathcal{A}_{2}} (\{2\}^{m_0},1,\{2\}^{m_1},3,\{2\}^{m_2},\ldots,\{2\}^{m_{2l-2}},1,\{2\}^{m_{2l-1}},3,\{2\}^{m_{2l}})  \\
&=(-1)^m
\biggl\{(-1)^l2^{1-2l}\binom{l+m}{l}-4 \binom{2 l+m}{2l}\biggr\} \beta_{4l+2m+1}\boldsymbol{p},
\end{split}
\end{equation}
\begin{equation} \label{MT2}
\begin{split}
&\sum_{\substack{ m_0+\cdots+m_{2l}=m \\ m_i\ge0\,(0\le i\le 2l) }}
\zeta_{\mathcal{A}_{2}}^{\star} (\{2\}^{m_0},1,\{2\}^{m_1},3,\{2\}^{m_2},\ldots,\{2\}^{m_{2l-2}},1,\{2\}^{m_{2l-1}},3,\{2\}^{m_{2l}})  \\
&=(-1)^l2^{1-2l}\binom{l+m}{l} \beta_{4l+2m+1}\boldsymbol{p}.
\end{split}
\end{equation}
\end{thm}
Saito-Wakabayashi's theorem (Theorem \ref{SWthm}) says that the sum of FMZ(S)Vs in $\mathcal{A}_1$ for the shuffle of $\{a,b\}^l$ and $\{c\}^m$ is zero for any odd positive integers $a, b$ and any even positive integer $c$. On the other hand, by our computer calculations, it seems that the similar sum of FMZ(S)Vs in $\mathcal{A}_2$ \emph{is not} a rational multiple of $\beta_{(a+b)l+cm+1}\boldsymbol{p}$, generally. For example, it is probable that $\zeta_{\mathcal{A}_2}(1,5,1,5)$ is not a rational multiple of $\beta_{13}\boldsymbol{p}$.

Zhao conjectures that the dimension of the $\mathbb{Q}$-vector space spanned by MZVs of weight $k$ coincides with the dimension of the $\mathbb{Q}$-vector space spanned by FMZVs in $\mathcal{A}_2$ of weight $k$ (\cite[Conjecture 9.6]{Zha15}). However, this conjecture doesn't mean that a correspondence $\zeta(k_1, \dots, k_r) \mapsto \zeta_{\mathcal{A}_2}(k_1, \dots, k_r)$ gives an isomorphism between these two spaces. In this situation, it is worth emphasizing that there exists a similarity between Bowman-Bradley type theorems for MZ(S)Vs and FMZ(S)Vs in $\mathcal{A}_2$, i.e. the sum of MZ(S)Vs for the shuffle of $\{1,3\}^l$ and $\{2\}^m$ is a rational multiple of $\pi^{4l+2m}$ and the similar sum of FMZ(S)Vs in $\mathcal{A}_2$ is a rational multiple of $\beta_{4l+2m+1}\boldsymbol{p}$. Note that these two rational coefficients are different.

We prove our main theorem in \S\ref{sec:Prel} and \S\ref{sec:Pr}.
%%%%%%%%%%%%%%%%%%%%%%%%%%%%%%%%%%%%%%%%%%%%%%%%%%%%%%%%%%%%%%%%%%%%%%%%%%%%%%%%%%%%%%%%%%%%%%%%%%%%%%
\section{Preliminaries} \label{sec:Prel}
We prepare some notation and lemmas in this section.
Let $\mathfrak{H}^1$ be the Hoffman algebra $\mathbb{Q}+\mathbb{Q}\langle x, y\rangle y$. 
We define two kinds of shuffle products $\sh$ and $\sha$ on $\mathfrak{H}^1$ as in \cite[\S2]{Mun09}. 
We call a tuple of positive integers an index. Let $\mathfrak{R}=\bigoplus_{r=0}^{\infty}\mathbb{Q}[\mathbb{Z}_{>0}^r]$ be the $\mathbb{Q}$-vector space spanned by all indices. Then, we use the same notation $\sh$ and  $\sha$ on $\mathfrak{R}$ by the correspondence $(k_1, \dots, k_r) \mapsto x^{k_r-1}y\cdots x^{k_1-1}y$ between $\mathfrak{R}$ and $\mathfrak{H}^1$.
Note that, for the definition of MZVs, the order of indices in \cite{Mun09} is reverse to ours.
For example, $(1,2)\sh(1)=3(1,1,2)+(1,2,1)$ and $(2,3)\sha(1)=(1,2,3)+(2,1,3)+(2,3,1)$.
Then, the summations in Theorem \ref{main} are written as $\zeta_{\mathcal{A}_{2}}((\{1,3\}^{l})\sha(\{2\}^{m}))$ and $\zeta_{\mathcal{A}_{2}}^{\star}((\{1,3\}^{l})\sha(\{2\}^{m}))$, respectively.
Here, we extend $\zeta_{\mathcal{A}_2}$ and $\zeta_{\mathcal{A}_2}^{\star}$ to functions on $\mathfrak{R}$, linearly.
\begin{lem} \label{muneta}
For non-negative integers $l$ and $m$, we have
\begin{multline*}
4^l\left\{(\{1,3\}^l)\sha (\{2\}^m)\right\} \\ =(\{2\}^{l+m})\sh (\{2\}^{l}) -\sum_{k=0}^{l-1}4^{k} \binom{2l+m-2k}{l-k}\left\{(\{1,3\}^k) \sha (\{2\}^{2l+m-2k})\right\}.
\end{multline*}
\end{lem}
\begin{proof}
This follows from \cite[Proposition 2 (1)]{Mun09}.
\end{proof}
The following lemma is the shuffle relation for FMZVs in $\mathcal{A}_2$.
\begin{lem} \label{shuffle}
For indices $\boldsymbol{k}$ and $\boldsymbol{l}=(l_1,\ldots,l_s)$, we have
\begin{align*}
&\zeta_{\mathcal{A}_2} (\boldsymbol{k} \sh \boldsymbol{l}) \\
&=(-1)^{l_1+\cdots+l_s} \sum_{\substack{e_1+\cdots +e_s=0, 1 \\ e_1,\ldots,e_s\ge0}}
\prod_{j=1}^s \binom{l_j+e_j-1}{e_j} \zeta_{\mathcal{A}_2} (\boldsymbol{k},l_s+e_s,\ldots,l_1+e_1) \boldsymbol{p}^{e_1+\cdots+e_s}.
\end{align*}
\end{lem}
\begin{proof}
This follows from \cite[Theorem 6.4]{SekD} which is also proved independently by Jarossay in \cite[Lemma 4.17]{Jar17} by taking $\varprojlim_n\mathcal{A}_n \twoheadrightarrow \mathcal{A}_2$.
\end{proof}
\begin{lem} \label{kkk}
For a positive integer $r$, we have
\begin{align}
\zeta_{\mathcal{A}_{2}} (\{2\}^{r}) 
&=(-1)^{r-1}2\beta_{2r+1} \boldsymbol{p}, \label{ZC}\\
\zeta_{\mathcal{A}_{2}}^{\star} (\{2\}^{r}) 
&=2\beta_{2r+1} \boldsymbol{p}.\label{ZC-star}
\end{align}
\end{lem}
\begin{proof}
The equality \eqref{ZC} is a special case of the second congruence in the last remark of \cite{ZC07}. The equality \eqref{ZC-star} is obtained by \eqref{ZC} and \cite[Corollary 3.16 (42)]{SS17}.
\end{proof}
\begin{lem}[{Hessami Pilehrood-Hessami Pilehrood-Tauraso \cite[Theorem 4.1]{HHT14}}] \label{two-three}
 For non-negative integers $a$ and $b$, we have
\[
\zeta_{\mathcal{A}_1} (\{2\}^{a},3,\{2\}^{b})
=\frac{(-1)^{a+b}2(a-b)}{a+1} \binom{2a+2b+3}{2b+2} \beta_{2a+2b+3}. 
\]
Here, we regard $\beta_{2a+2b+3}$ as an element of $\mathcal{A}_1$ by the projection $\mathcal{A}_2 \twoheadrightarrow \mathcal{A}_1$.
\end{lem}
\begin{lem} \label{aaa}
For non-negative integers $l$ and $m$ with $(l,m)\neq(0,0)$, we have
\begin{align*}
\zeta_{\mathcal{A}_{2}}((\{2\}^{l+m})\sh(\{2\}^{l}))
=(-1)^{m} 2 \left\{1-2\binom{4l+2m}{2l}\right\}
\beta_{4l+2m+1} \boldsymbol{p}. 
\end{align*}  
\end{lem}
\begin{proof}
By Lemma \ref{shuffle}, \ref{kkk} \eqref{ZC}, and \ref{two-three}, we have
\begin{align*}
&\zeta_{\mathcal{A}_{2}}((\{2\}^{l+m})\sh(\{2\}^{l}))\\
&=\zeta_{\mathcal{A}_2}(\{2\}^{2l+m})+2\sum_{j=0}^{l-1}\zeta_{\mathcal{A}_2}(\{2\}^{l+m+j}, 3, \{2\}^{l-j-1})\boldsymbol{p}\\
&=(-1)^{m-1}\left\{2\beta_{4l+2m+1}\boldsymbol{p}+4 \sum_{j=0}^{l-1}\frac{m+2j+1}{l+m+j+1}
\binom{4l+2m+1}{2l-2j} \beta_{4l+2m+1}\boldsymbol{p}\right\}.
\end{align*} 
Since $\frac{a-2b}{a}\binom{a}{b}=\binom{a-1}{b}-\binom{a-1}{b-1}$, 
by putting $a=4l+2m+2$ and $b=2j$, 
we have
\begin{align*}
&\sum_{j=0}^{l}\frac{m+2j+1}{l+m+j+1} \binom{4l+2m+1}{2l-2j}=\sum_{j=0}^{l}\frac{2l+m-2j+1}{2l+m+1} \binom{4l+2m+2}{2j}\\
&=\sum_{j=0}^{l}\left\{\binom{4l+2m+1}{2j}-\binom{4l+2m+1}{2j-1}\right\}
=\sum_{j=0}^{2l}(-1)^{j}\binom{4l+2m+1}{j}\\
&=\sum_{j=0}^{2l}(-1)^{j}\left\{\binom{4l+2m}{j}+\binom{4l+2m}{j-1}\right\}
=\binom{4l+2m}{2l}.
\end{align*} 
Hence, we obtain the desired formula.
\end{proof}
\begin{lem} \label{vandermonde}
For non-negative integers $l$ and $m$, we have
\begin{align*}
&\sum_{k=0}^{l} (-1)^{k} \binom{2l+m-2k}{l-k}\binom{2l+m-k}{k}=1, \\
&\sum_{k=0}^{l} 4^{k} \binom{2l+m-2k}{l-k}\binom{2l+m}{2k}=\binom{4l+2m}{2l}.
\end{align*}
\end{lem}
\begin{proof}
Since $\binom{a-b}{c-b}\binom{a}{b}=(-1)^{a-c}\binom{c}{b}\binom{-c-1}{a-c}$, by putting $a=2l+m-k$, $b=k$, and $c=l$, we have
\begin{align*}
&\sum_{k=0}^{l} (-1)^{k} \binom{2l+m-2k}{l-k}\binom{2l+m-k}{k}\\
&=(-1)^{l+m}\sum_{k=0}^{l+m}\binom{l}{k}\binom{-l-1}{l+m-k}
=(-1)^{l+m}\binom{-1}{l+m}=\binom{l+m}{l+m}=1
\end{align*}
by the Chu-Vandermonde identity. Next, we prove the second equality. 
Let $\binom{n}{a,b,c}:=n!/(a!b!c!)$. 
Since
\[
(1+Y)^{4l+2m}= (1+2Y+Y^2)^{2l+m}=\sum_{\substack{a+b+c=2l+m\\a,b,c\ge0}}\binom{2l+m}{a,b,c}(2Y)^bY^{2c}
\]
holds, by comparing the coefficient of $Y^{2l}$, we have
\[
\binom{4l+2m}{2l}=\sum_{j=0}^l\binom{2l+m}{j+m,2l-2j,j}2^{2l-2j}=\sum_{k=0}^{l} 4^{k} \binom{2l+m-2k}{l-k}\binom{2l+m}{2k}.
\]
This concludes the proof.
\end{proof}
\begin{lem} \label{Yam-thm}
For non-negative integers $l$ and $m$, we have
\begin{align*}
&\zeta_{\mathcal{A}_2}^{\star}((\{1,3\}^l)\sha (\{2\}^m))\\
&=\sum_{\substack{2i+k+u=2l \\ j+n+v=m}}(-1)^{j+k}\binom{k+n}{k}\binom{u+v}{u}\zeta_{\mathcal{A}_2}((\{1,3\}^i)\sha (\{2\}^j))\zeta_{\mathcal{A}_2}^{\star}(\{2\}^{k+n})\zeta_{\mathcal{A}_2}^{\star}(\{2\}^{u+v}),
\end{align*}
where parameters $i, j, k, n, u, v$ are non-negative integers.
\end{lem}
\begin{proof}
This follows from \cite[Theorem  2.1]{Yam13}.
\end{proof}
%%%%%%%%%%%%%%%%%%%%%%%%%%%%%%%%%%%%%%%%%%%%%%%%%%%%%%%%%%%%%%%%%%%%%%%%%%%%%%%%%%%%%%%%%%%%%%%%%%%%%%
\section{Proof of the main theorem} \label{sec:Pr}
\begin{proof}[Proof of Theorem \ref{main}]
First, we prove \eqref{MT1} by induction on $l$. 
We see that the case $l=0$ holds by Lemma \ref{kkk} \eqref{ZC}. 
For the general case, let $l$ be a positive integer and $m$ a non-negative integer.
By Lemma \ref{muneta}, we have
\begin{multline*}
\zeta_{\mathcal{A}_2} ((\{1,3\}^l)\sha (\{2\}^m)) \\
=4^{-l}\zeta_{\mathcal{A}_2} ((\{2\}^{l+m})\sh (\{2\}^{l})) -\sum_{k=0}^{l-1}4^{k-l} \binom{2l+m-2k}{l-k} \zeta_{\mathcal{A}_2} ((\{1,3\}^k) \sha (\{2\}^{2l+m-2k})). 
\end{multline*} 
Hence, by Lemma \ref{aaa} and the induction hypothesis, we have
\begin{align*}
&\zeta_{\mathcal{A}_2} ((\{1,3\}^l)\sha (\{2\}^m)) \\
&=(-1)^m 2^{1-2l}\left\{1-2\binom{4l+2m}{2l}\right\}
\beta_{4l+2m+1} \boldsymbol{p} \\
&\quad -\sum_{k=0}^{l-1} 4^{k-l}\binom{2l+m-2k}{l-k} \cdot (-1)^m
\biggl\{(-1)^k2^{1-2k}\binom{2l+m-k}{k} -4\binom{2 l+m}{2k} \biggr\}
\beta_{4l+2m+1} \boldsymbol{p}.
\end{align*}
By Lemma \ref{vandermonde}, we can simplify as 
\begin{align*}
\zeta_{\mathcal{A}_2} ((\{1,3\}^l)\sha (\{2\}^m))
&=(-1)^m2^{1-2l}\left\{1-2\binom{4l+2m}{2l}\right\}
\beta_{4l+2m+1} \boldsymbol{p} \\
&\quad-(-1)^{m}2^{1-2l}\left\{1-(-1)^{l}\binom{l+m}{l}\right\}\beta_{4l+2m+1} \boldsymbol{p}\\
&\quad+(-1)^m4^{1-l}\left\{\binom{4l+2m}{2l}-4^{l}\binom{2l+m}{2l}\right\}
\beta_{4l+2m+1} \boldsymbol{p} \\
&=(-1)^m\biggl\{(-1)^l2^{1-2l}\binom{l+m}{l}-4\binom{2 l+m}{2l}\biggr\} \beta_{4l+2m+1} \boldsymbol{p}.
\end{align*} 
Next, we prove \eqref{MT2}. By the equality \eqref{MT1} and Lemma \ref{kkk} \eqref{ZC-star}, many terms in the right-hand side of the equality in Lemma \ref{Yam-thm} vanish and we have
\begin{align*}
&\zeta_{\mathcal{A}_2}^{\star} ((\{1,3\}^l)\sha (\{2\}^m)) \\
&=(-1)^m\zeta_{\mathcal{A}_2}((\{1,3\}^l)\sha (\{2\}^m)) + 2\binom{2l+m}{2l}\zeta_{\mathcal{A}_2}^{\star}(\{2\}^{2l+m})\\
&=(-1)^l2^{1-2l}\binom{l+m}{l} \beta_{4l+2m+1}\boldsymbol{p}.
\end{align*}
This finishes the proof.
\end{proof}
%%%%%%%%%%%%%%%%%%%%%%%%%%%%%%%%%%%%%%%%%%%%%%%%%%%%%%%%%%%%%%%%%%%%%%%%%%%%%%%%%%%%%%%%%%%%%%%%%%%%%%%%%
\section*{Acknowledgements}
The authors would like to thank Doctor Hisatoshi Kodani for valuable comments. They also would like to express their gratitude to the anonymous referee for useful suggestions.
%%%%%%%%%%%%%%%%%%%%%%%%%%%%%%%%%%%%%%%%%%%%%%%%%%%%%%%%%%%%%%%%%%%%%%%%%%%%%%%%%%%%%%%%%%%%%%%%%%%%%%%%%


\begin{thebibliography}{99}
\bibitem{BB02} D.~Bowman and D.~M.~Bradley:
\textit{The algebra and combinatorics of shuffles and multiple zeta values}, 
J.~Combin.~Theory Ser.~\textbf{A 97} (2002), 43--61.

\bibitem{HHT14} KH.~Hessami Pilehrood, T.~Hessami Pilehrood, and R.~Tauraso:
\textit{New properties of multiple harmonic sums modulo $p$ and $p$-analogues of Leshchiner's series}, 
Trans.~Amer.~Math.~Soc.~\textbf{366} (2014), no.~6, 3131--3159.

\bibitem{Jar17} D.~Jarossay:
\textit{An explicit theory of $\pi_{1}^{\text{un,crys}}(\mathbb{P}^{1} - \{0,\mu_{N},\infty\})$-II-1: Standard algebraic equations of prime weighted multiple harmonic sums and adjoint multiple zeta values}, 
arXiv:1412.5099v3.

\bibitem{KST12} H.~Kondo, S.~Saito, and T.~Tanaka:
\textit{The Bowman-Bradley theorem for multiple zeta-star values}, 
J.\ Number Theory \textbf{132} (2012), 1984--2002.

\bibitem{Mun09} S.~Muneta:
\textit{A note on evaluations of multiple zeta values}, 
Proc.~Amer.~Math.~Soc.~\textbf{137} (2009), 931--935. 

\bibitem{Ros15} J.~Rosen:
\textit{Asymptotic relations for truncated multiple zeta values}, 
J.\ Lond.\ Math.\ Soc.\ (2) \textbf{91} (2015), 554--572. 

\bibitem{SW16} S.~Saito and N.~Wakabayashi:
\textit{Bowman-Bradley type theorem for finite multiple zeta values}, 
Tohoku Math.~J. \textbf{68} (2016), 241--251.

\bibitem{SS17}
K.~Sakugawa and S.~Seki:
\textit{On functional equations of finite multiple polylogarithms},
J.~Algebra \textbf{469} (2017), 323--357.

\bibitem{SekD} S.~Seki:
\textit{Finite multiple polylogarithms}, 
Doctoral dissertation in Osaka University, 2017.

\bibitem{Sek16} S.~Seki:
\textit{The $\boldsymbol{p}$-adic duality for the finite star-multiple polylogarithms},
to appear in Tohoku Math.~J.

\bibitem{Yam13} S.~Yamamoto:
\textit{Explicit evaluation of certain sums of multiple zeta-star values},
Funct.~Approx.~Comment.~Math.~\textbf{49} (2) (2013), 283--289.

\bibitem{Zha15} J.~Zhao:
\textit{Finite multiple zeta values and finite Euler sums},
arXiv:1507.04917.

\bibitem{ZC07} X.~Zhou and T.~Cai: 
\textit{A generalization of a curious congruence on harmonic sums}, 
Proc.~Amer.~Math.~Soc.~\textbf{135} (2007), 1329--1333.
\end{thebibliography}
\end{document}